\documentclass[preprint]{imsart}

\RequirePackage[OT1]{fontenc}
\RequirePackage{amsthm,amsmath,amssymb}
\RequirePackage[numbers]{natbib}
\RequirePackage[colorlinks,citecolor=blue,urlcolor=blue]{hyperref}
\usepackage{stmaryrd}
\usepackage[sans]{dsfont}

\newcommand{\trans}{\mathsf{T}}
\newcommand{\e} {\varepsilon}

\newcommand{\di}[1]{\operatorname{d}\!#1}

\newcommand{\p} {\mathbf{P}}

\newcommand{\R} {\mathcal{R}}
\newcommand{\I} {\operatorname{Im}}
\newcommand{\op} {\operatorname}
\newcommand{\Rea} {\mathcal{R}}

\newcommand{\pro} {\mathcal{P}}

\startlocaldefs
\theoremstyle{plain}

\newtheorem{definition}{Definition}[section]
\newtheorem{lemma}{Lemma}[section]
\newtheorem{theorem}{Theorem}[section]

\endlocaldefs

\begin{document}

\begin{frontmatter}
\title{Computation of open-loop inputs for uniformly ensemble controllable systems}

\runtitle{Computation of inputs for ensemble controllable systems}


\begin{aug}
\author{\fnms{Michael} \snm{Sch\"onlein}\ead[label=e1]{schoenlein@mathematik.uni-wuerzburg.de}}

\runauthor{M. Sch\"onlein}

\affiliation{University of W\"urzburg}

\address{Institute for Mathematics\\
         University of W\"urzburg\\
         Emil-Fischer Stra\ss e 40\\
         97074 W\"urzburg, Germany\\
\printead{e1}\\
\phantom{E-mail:\ }
}
\end{aug}

\begin{abstract}
This paper presents computational methods for families of linear systems depending on a parameter. Such a family is called ensemble controllable if for any family of parameter-dependent target states and any neighborhood of it there is a parameter-independent input steering the origin into the neighborhood.  Assuming that a family of systems is ensemble controllable we present methods to construct suitable open-loop input functions. Our approach to solve this infinite-dimensional task is based on a combination of methods from the theory of linear integral equations and finite-dimensional control theory.

\end{abstract}



This work is accepted for publication in \emph{Mathematical Control and Related Fields}
and was supported by the German Research Foundation (DFG) \\ grant SCHO 1780/1-1.\\[2ex]
Received 30. September 2019; revised 30. August 2020.

\end{frontmatter}

\section{Introduction}

Recently, the topic of ensemble control has become an emerging field in mathematical systems and control theory. An essential characteristic is the task of controlling a large, potentially infinite, number of states, or systems at once, using a single open-loop input function or a single feedback controller. Although the notion {\it ensemble control} is relatively new, such problems have been studied in the 1970s and 1980s under different names. For instance, in \cite{Tannenbaum1981invariance} the {\it blending problem} refers to the task of stabilizing a family of scalar transfer functions by a common dynamic feedback compensator. Also, in a series of papers Ghosh and co-workers consider the problem of simultaneous controllability and stabilizability for families of linear systems, cf. \cite{ghosh1985,ghosh1986_I,ghosh1988_II,ghosh1986_partial_pp}. 

In principle, ensemble control embraces several facets. A current nice exposition for the applicability of ensemble control is given in \cite[Section~2.4]{brockett2012notes}.  For instance, it contains problems where states or observations of a system are subject to uncertainties and only a probability distribution on the state space or output snapshots are available. In this case, the ensemble control problem  is to morph a given initial probability density function of the states into a desired one by transporting it along a linear system. This yields controllability problems for the Liouville equation and the Fokker-Plank equation.  We refer to, e.g. \cite{brockett2012notes,chen2017optimal,DIRR20161018,fleig2016estimates,fuhrmann2015mathematics} for recent works on this topic and note that this list is not intended to be complete. Also, in \cite{shen2017discrete,zeng2016ensemble} the problem of reconstructing an initial probability density function from output snapshots is addressed.

Moreover, ensemble control also deals with families of systems, usually defined by parameter-dependent systems. Here, the goal is to control the entire family by finitely many open-loop inputs or feedback controllers. That is, given families of initial states and desired terminal states the task is to steer the initial family towards the desired family in some finite time by a parameter-independent input function. Problems of his kind arise for instance in robotics and quantum control, cf. \cite{becker2012approximate,li2006control,turinici2004optimal}. Of course, characterizations of ensemble controllability depend on the structure of the parameter-dependent system. For necessary and sufficient conditions for ensemble controllability for linear systems, bilinear systems and nonlinear systems we refer to \cite{dirr2018uniform,helmke2014uniform,li_tac_2016,li2011,schonlein2016controllability,Zeng_scl_2016}  and  \cite{belhadj2015ensemble,Gauthier_scl_2018,dirr2012ensemble} and \cite{Agrachev_ensemble_lie_2016}, respectively.

In this paper we consider ensembles defined by families of parameter-dependent linear systems. Assuming that it is ensemble controllable and given a family of target states, we tackle the question how to derive suitable open-loop input functions.  The rest of the paper is organized as follows. In Section~\ref{sec:problem} we outline the basic problem studied in this paper and provide the necessary notations and definitions concerning ensemble controllability. In Section~\ref{sec:collocation} we extend a known discretization method to obtain approximate solutions to vector-valued linear integral equations of first kind. To get such a collocation method, we provide the reproducing kernel structure for Hilbert spaces of vector-valued functions. Moreover, we derive a convergence rate under the weaker regularity assumption that the reproducing kernel is Lipschitz continuous. Note that the standard result by Nashed and Wahba \cite[Theorem~3.1]{nashed_math_comp_1974} requires that the reproducing kernel satisfies a differentiability condition. In Section~\ref{sec:consensus} we present new methods that provide algorithms for the computation of suitable open-loop input functions.

\section{Problem statement and computational method}\label{sec:problem}

In this paper we consider families of parameter-dependent linear control systems 
\begin{equation}\label{eq:sys-par} 
\begin{split}
\tfrac{\partial x}{\partial t}
(t,\theta)&=A(\theta)x(t,\theta)+B(\theta)u(t)\\
x(0,\theta)&=x_0(\theta)\in\mathbb{R}^n,
\end{split}
\end{equation}
where the matrices $A(\theta)\in\mathbb{R}^{n\times n}$ and $B(\theta)\in\mathbb{R}^{n\times m}$ depend continuously on a parameter $\theta$ which is varying over a nonempty compact interval $\p \subset \mathbb{R}$. Throughout the paper we will use the short notation $(A,B) \in C_{n,n}(\p) \times C_{n,m}(\p)$ for $A \in C(\p,\mathbb{R}^{n \times n})$ and $B \in C(\p,\mathbb{R}^{n \times m})$. Similarily we shortly write $L^q_m([0,T])$ instead of $L^q([0,T],\mathbb{R}^m)$. For $u\in L^1_m([0,T])$ let $\varphi\big(t,u,x_0(\theta),\theta\big)$ denote the solution to \eqref{eq:sys-par}. To express that \eqref{eq:sys-par} defines an infinite-dimensional system whose states are functions mapping the compact interval $\p$ to $\mathbb{R}^n$ we write 
$\varphi(t,u,x_0)(\theta) :=\varphi\big(t,u,x_0(\theta),\theta\big)$.

The precise definition of ensemble controllability is as follows. A pair $(A,B) \in C_{n,n}(\p)\times C_{n,m}(\p)$ is called \emph{uniformly ensemble controllable} if for every $\e>0$ and for every $f \in C_n(\p)$ there are $T>0$ and $u \in L^1_m([0,T])$ such that
 \begin{equation*}
\|\varphi(T,u,x_0) -f\|_\infty =\sup_{\theta \in \p} \| \varphi(T,u,x_0(\theta),\theta) - f(\theta) \|_{\mathbb{R}^n} < \e.
\end{equation*}
As shown in \cite[Thm.~3.1.1, Remark~3.1.2]{triggiani75} in continuous-time the final time $T>0$ can be chosen arbitrary. Thus, from now we fix $T>0$. We note that for discrete-time systems the final time $T>0$ depends on the target state $f$ and the desired accuracy $\e>0$. 

Most of the literature mentioned in the introduction considers the problem of deriving convenient necessary and sufficient conditions for uniform ensemble controllability in terms of $(A,B)$. Almost all results are existence results, i.e. the existence of an appropriate input function is guaranteed but explicit expressions are not available in general. Here, we tackle the next step, namely supposing a pair $(A,B)$ is uniformly ensemble controllable, this paper considers the problem of determining suitable input functions.

{In order to make use of Hilbert space properties we are looking for input functions in $L^2_m([0,T])$.} If the pair $(A,B) \in C_{n,n}(\p) \times C_{n,m}(\p)$ is uniformly ensemble controllable, the {range of the} compact reachability integral operator 
 \begin{equation}\label{eq:def_reachability_operator}
\R \colon L^2_m([0,T]) \to L^2_n(\p), \quad \R u (\theta) = \int_0^T e^{A(\theta)(T-s)} B(\theta) \, u(s) \di{s}
\end{equation}
is dense in $C_n(\p)$, cf. \cite{dirr2018uniform}.  We consider how to determine an appropriate input function $u \in  L^2_m([0,T]) $ such that for given $f \in C_n(\p)$  and $\e >0$ one has                               
 \begin{equation}\label{eq:basic-problem-moment-collo}
\| \R u - f\|_{C_n(\p)}< \e.
\end{equation}
In other words, the problem of computing an appropriate open-loop input function is equivalent to finding an approximate solution to the linear operator equation
 \begin{equation}\label{eq:lin_op_eq}
 \R u = f.
\end{equation}
This is a classical ill-posed inverse problem. Well-known necessary and sufficient conditions such that \eqref{eq:lin_op_eq} has a solution is given by Picard's criterion, cf.~\cite{engl_etal_1996,kirsch2011introduction,kress}. It also provides an explicit representation of the solution. This, however, has the drawback that it depends on the singular system of the compact reachability operator $\mathcal{R}$. The results in  \cite{li2011} are based on Picard's criterion. Other related approaches can be found in~\cite{TIE20173051,zeng2018computation}.

In this paper we follow a different approach, namely a discretization method. {Based on the assumption that} $(A,B) \in C_{n,n}(\p) \times C_{n,m}(\p)$ {is uniformly} ensemble controllable, we have 
\begin{align*}
\overline{\I\R} = C_n(\p) \qquad \text{ and } \qquad (\I\R)^\bot = \{0\}.
\end{align*}
The uniform ensemble reachability assumption yields that for every $f\in C_n(\p)$ and every $\e >0$ there is a function $g \in \operatorname{Im}\operatorname{\R}$ such that $\|f-g\|_\infty<\e$.  Thus, without loss of generality we treat the case that $f$ is an element of the range of $\R$. Despite the regularization techniques and projection methods available in the literature above, we follow a different approach {called moment collocation, which inherits well-known techniques from control theory, e.g. the controllability Gramian. To obtain computational methods we will incorporate} ideas originating from decentralized methods to solve linear finite-dimensional equations, as presented, e.g. in~\cite{anderson2015decentralized,mou2013fixed,shi2017network}.
Finally, we note that our subsequent methods do not depend on the singular system of the reachability operator $\Rea$, which is used in \cite{li2011}.

\section{Moment collocation}\label{sec:collocation}

In this section we present {the moment collocation} method to derive approximate solutions to the linear equation \eqref{eq:lin_op_eq}. The basic idea is simply to take $N \in \mathbb{N}$ different parameters out of the compact parameter interval $\p=[\theta^-,\theta^+]$, i.e. we pick
\begin{equation*}
\theta^-= \theta_1 < \theta_2 < \cdots < \theta_N=\theta^+.
\end{equation*}
The $\theta_i$'s are sometimes called \emph{moments}. For the pair $(A,B)\in C_{n,n}(\p) \times C_{n,m}(\p)$ evaluated at the moments we use the short notation
\begin{equation*}
(A_k,B_k) := (A(\theta_k) , B(\theta_k) ) , \quad k=1,...,N.
\end{equation*}
Then, we define the interpolation operator $\mathcal{I}_N \colon L^2_m([0,T]) \to \mathbb{R}^{Nn}$,
 \begin{equation*}
\mathcal{I}_N u  = \begin{pmatrix}
\int_0^T e^{A_1(T-s)} B_1 \, u(s) \di{s}\\
\vdots\\
\int_0^T e^{A_N(T-s)} B_N \, u(s) \di{s}
                  \end{pmatrix}.
\end{equation*}
A justification for the term interpolation operator can be found e.g. in  \cite[Section~11.2]{kress}.  
As in the proof of \cite[Theorem~3~(c)]{dirr2018uniform}, the application of the restriction property to $\{\theta_1,..., \theta_N\}$ yieds that the parallel connection 
 \begin{equation}\label{eq:collo-def-parallel-connetion}
\left(
\begin{pmatrix}
A_1 &  & \\
  & \ddots & \\
   &  &        A_N\end{pmatrix},
\begin{pmatrix}
B_1\\
\vdots\\
   B_N        \end{pmatrix}
   \right)
\end{equation}
of the finite-dimensional pairs $(A_k,B_k)$, $k=1,...,N$ is controllable. Thus, for each $f \in C_n(\p)$ and
 \begin{equation*}
F_N = \begin{pmatrix}
f(\theta_1)\\
\vdots\\
f(\theta_N)        
\end{pmatrix}
\end{equation*}
the linear equation
\begin{equation}\label{eq:lin-eq-RNu=FN}
\mathcal{I}_N u = F_N 
\end{equation}
has a unique minimal $L^2$-norm solution, denoted by $u_{\parallel N}$. It is well-known, cf. \cite[Section~3.5, Theorem~5]{sontag} that the solution can be written as
\begin{equation*}\label{eq:u_parallel}
u_{\parallel N}(s) = \begin{pmatrix}
B_1^\trans & \cdots &    B_N^\trans \end{pmatrix} \begin{pmatrix}
e^{A_1^\trans(T-s)} &  & \\
  & \ddots & \\
   &  &        ^{A_N^\trans(T-s)}\end{pmatrix}
   W^{-1}_\parallel F_N,
\end{equation*}
where $W_\parallel$ denotes the controllability Gramian of the parallel connetion \eqref{eq:collo-def-parallel-connetion} given by
\begin{equation}\label{eq:def-Gramian_parallel}
 \begin{split}
\int_0^T
\begin{pmatrix}
e^{A_1(T-s)} B_1 & \cdots &  e^{A_N(T-s)}  B_N
\end{pmatrix}
\begin{pmatrix}
B_1^\trans e^{A_1^\trans(T-s)}& \cdots &    B_N^\trans e^{A_N^\trans(T-s)}
\end{pmatrix} 
\di{s}\\
= \left(\int_0^T  e^{A_l(T-s)} B_l  B_k^\trans e^{A_k^\trans(T-s)}      \di{s}  \right)_{l,k=1,...,N}. \end{split}
\end{equation}

\medskip
In the following we are interested in, whether for given $f \in C_n(\p)$ and $\e>0$, the input $u_{\parallel N}$ can be used to obtain a suitable input satisfying \eqref{eq:basic-problem-moment-collo} by choosing suitable moment parameters $\{\theta_1,...,\theta_N\}$. To this end, we define
\begin{equation*}
\Delta_N:= \sup_{\theta \in \p} \left( \inf_{i=1,...,N} |\theta_i - \theta| \right).
\end{equation*}
Note, if the moments are chosen equidistantly, i.e. for $\delta>0$ and $N \in \mathbb{N}$ the moments
\begin{equation*}
 \theta^- = \theta_1 \leq \cdots < \theta_N  = \theta_+  \quad \text{ are such that } \quad   | \theta_{i+1} - \theta_i| = \delta,
\end{equation*}
it holds 
$$ \Delta_N \leq \frac{\delta}{2}.$$

{In order to carry out the convergence analysis we need some preparations. In doing so, we introduce Hilbert spaces with reproducing kernels. Let $X_n(\p)$ denote a Hilbert space of functions from $\p$ to $\mathbb{R}^n$.} The following definition is taken from \cite{pedrick} and is in line with the definition in \cite{nashed_math_comp_1974} which treats the scalar case.

\begin{definition}[Reproducing kernel] 
A reproducing kernel for the Hilbert space $X_n(\p)$ is a family $Q(\theta,\tilde \theta)$ of transformations from $\mathbb{R}^n$ to $\mathbb{R}^n$ defined for each $\theta, \tilde \theta \in \p$ and having the properties:
\begin{enumerate}
 \item for every $\tilde \theta \in \p$ and $v \in \mathbb{R}^n$ it holds $ Q(\cdot, \tilde \theta)v \in X_n(\p)$.
 \item for every $\tilde \theta \in \p, f \in X_n(\p)$ and $v \in \mathbb{R}^n$ it holds
 \begin{equation*}
\langle v,f(\tilde \theta) \rangle_{\mathbb{R}^n} = \langle Q(\cdot,\tilde \theta )v ,f \rangle_{X_n(\p)}.
\end{equation*}
\end{enumerate}
\end{definition}
\medskip

{To get a convergence result, we equip the range $\I\R$ with an inner product}. Let  $\R^\sharp \colon \I\R \to L^2_m([0,T])$ denote the Moore-Penrose inverse of $\R$, that is for $f \in \I\R$ the element $\R^\sharp f \in L^2_m([0,T])$ is the unique solution of $\R u=f$ with minimal $ L^2_m([0,T])$-norm, cf. \cite{engl_etal_1996}. Then, an inner product on $  \I \R$ is given by
 \begin{equation*}
\langle f,g \rangle_{\I\R} := \langle \R^\sharp f, \R^\sharp g \rangle_{L^2_m([0,T])} = \int_0^T \langle \R^\sharp f (s), \R^\sharp g(s) \rangle_{\mathbb R^m} \di{s} \qquad f,g \in \I\R.
\end{equation*}

\begin{lemma}
$(\I\R, \langle\cdot,\cdot\rangle_{\I \R}) $ is a reproducing kernel Hilbert space.
\end{lemma}
\medskip
{\it Proof.} {That $\I\R$ defines a Hilbert space is shown in \cite[Prop.~2.1]{nashed_math_comp_1974}.} To show the second claim, we verify that 
 \begin{equation*}
Q(\theta, \tilde \theta) = \int_0^T e^{A(\theta)(T-s)} B(\theta) B(\tilde \theta)^\trans\,e^{A(\tilde \theta)^\trans(T-s)} \di{s}
\end{equation*}
defines a reproducing kernel. 

To conclude the first property, let $v\in \mathbb{R}^n$ and $\tilde \theta \in \p$. Then, by definition we have
 \begin{equation*}
Q(\theta, \tilde \theta)v = \int_0^T e^{A(\theta)(T-s)} B(\theta) B(\tilde \theta)^\trans\,e^{A(\tilde \theta)^\trans(T-s)}v \di{s}.
\end{equation*}
Since $s \mapsto B(\tilde \theta)^\trans\,e^{A(\tilde \theta)^\trans(T-s)}v $ is an $L^2_m([0,T])$-function, the first property follows.

To conclude the second property take $f \in \I\R$ and let $u_f:= \R^\sharp f \in L^2_m([0,T])$. Then, for every $v \in \mathbb{R}^n$ and $\tilde \theta \in \p$ one has on one hand
 \begin{equation*}
\langle v,f(\tilde \theta) \rangle_{\mathbb{R}^n} =  \int_0^T v^\trans \, e^{A(\tilde \theta)(T-s)} B(\tilde \theta) u_f(s) \di{s}
\end{equation*}
and on the other hand
 \begin{equation*}
 \begin{split}
\langle Q(\cdot,\tilde \theta )v ,f \rangle_{\I\R} &= \langle \R^\sharp Q(\cdot,\tilde \theta )v , \R^\sharp f \rangle_{L^2_m([0,T])} \\
&= \int_0^T \langle  B(\tilde \theta)^\trans\,e^{A(\tilde \theta)^\trans(T-s)}v ,u_f(s)  \rangle_{\mathbb{R}^n}  \di{s}\\
&= \int_0^T v^\trans \, e^{A(\tilde \theta)(T-s)} B(\tilde \theta) u_f(s) \di{s}.
\end{split}
\end{equation*}
The assertion follows by comparing the right hand sides.
\hfill $\Box$

\medskip
%
%
%
Next, we present the essential structure given by the reproducing kernel. Let $\operatorname{e}_i$ denote the $i$th standard basis vector in $\mathbb{R}^n$ and  
\begin{align*}
U := \overline{\operatorname{span}\{ u_{i,\theta}= B(\theta)^\trans e^{A(\theta)^\trans(T-\cdot)} \operatorname{e}_i\, \, | \, \, i=1,...,n\, , \theta \in \p\} } \subset L_m^2([0,T])
\end{align*}
and note that
\begin{align*}
\I \Rea = \operatorname{span}\{  Q(\cdot,\tilde \theta) \operatorname{e}_i \, | \, i=1,...,n, \, \tilde \theta \in \p\}.
\end{align*}
Then, by the reproducing kernel structure we get
 \begin{equation*}
 \begin{split}
\langle \Rea u_{i,\tilde \theta} ,\Rea u_{j,\hat \theta} \rangle_{\I\R} &=
 \langle Q(\cdot,\tilde \theta )\operatorname{e}_i ,  Q(\cdot,\hat \theta )\operatorname{e}_j\rangle_{\I\R} = \langle \operatorname{e}_i, Q(\tilde \theta, \hat \theta )\operatorname{e}_j \rangle_{\mathbb{R}^n} \\
&= \int_0^T \langle  u_{i,\tilde \theta} (s) ,u_{j,\hat \theta}(s)  \rangle_{\mathbb{R}^n} \di{s}
= \langle  u_{i,\tilde \theta} ,u_{j,\hat \theta}  \rangle_{L_m^2([0,T])}.
\end{split}
\end{equation*}
That is, $\Rea \colon U \to \I \Rea$ is an isometric isomorphism and using $\Rea \Rea^\sharp f =f$ for all $f \in \I \Rea$, we have 
\begin{equation}\label{eq:convergence_reproducing_witz}
\begin{split}
 \|\R u - f \|_{\operatorname{Im}\R}^2 &= \langle \Rea u ,\Rea u \rangle_{\I\R} - 2 \langle \Rea u ,\Rea \Rea^\sharp f \rangle_{\I\R} +\langle f ,f \rangle_{\I\R}\\
 &=  \langle  u , u \rangle_{L^2_m([0,T])} - 2 \langle  u , \ \Rea^\sharp f\rangle_{L^2_m([0,T])} +\langle \Rea^\sharp f ,\Rea^\sharp f \rangle_{L^2_m([0,T])} \\
 &=  \|u -\Rea^\sharp  f\|_{L^2_m([0,T])}^2
 \end{split}
\end{equation}
for all $u \in U$.

The significance here is that the convergence analysis of the moment collocation method in the Hilbert space $L^2_m([0,T])$ can be traced back to a convergence analysis in $(\operatorname{Im}\R, \langle\cdot,\cdot\rangle_{\I \R}) $.
%
%
%
Next, we construct the solution $u_{\parallel N}$ by following a classical approach in the theory of integral equations, cf. \cite{kress,nashed_math_comp_1974}. To do so, we take finite-dimensional subspaces of $\I\Rea$ by choosing for $i \in \{1,...,n\}$ and $k \in \{1,...,N\}$  the $i$th column of the matrix $B_k^\trans e^{A^\trans_k(T-s)} $ as an element of $L^2_m([0,T])$. That is, we define
\begin{align*}
v_{i,k}(\cdot) := Q(\cdot,\theta_k) \operatorname{e}_i,  \quad u_{i,k}(\cdot) := B(\theta_k)^\trans e^{A(\theta_k)^\trans(T-\cdot)} \operatorname{e}_i
\end{align*}
and
\begin{align*}
V_{N} &:= \operatorname{span}\{ v_{1,k},...,v_{n,k}\, \, | \, \, k=1,...,N\} \subset \I\Rea \\ 
U_{N} &:= \operatorname{span}\{ u_{1,k},...,u_{n,k}\, \, | \, \, k=1,...,N\} \subset L_m^2([0,T]).
\end{align*}
We note explicitly that, by definition it holds
\begin{align}\label{eq:Ruik-vik}
\Rea u_{i,k} = v_{i,k} \qquad \text{ and } \qquad u_{i,k} = \Rea^\sharp v_{i,k}
\end{align}
for all $i=1,...,n$   and $ k =1,...,N$. Without loss of generality, we assume that the functions $v_{i,k}$ are linearly independent. Consequently, $U_N$ and $V_N$ are  $nN$-dimensional subspaces of $L^2_m([0,T])$ and $\I \R$, respectively.

%
%

Now, we return to the solution $u_{\parallel N}$ of \eqref{eq:lin-eq-RNu=FN} and shall show that it lies in $U_N$. To see this, let $u \in U_N$, i.e. there are real numbers $\alpha_{1,1},...,\alpha_{n,N}$ such that 
\begin{equation}\label{eq:def-u-alpha-v}
 u= \sum_{k=1}^{N} \sum_{i=1}^{n} \alpha_{i,k} u_{i,k}.
\end{equation}
Let $\Rea_k \colon L^2_m([0,T]) \to \mathbb{R}^{n}$,
 \begin{equation*}
\Rea_k u  =
\int_0^T  e^{A_k(T-s)} B_k\, u(s) \di{s}.
\end{equation*}
Then, the linear equation \eqref{eq:lin-eq-RNu=FN} can also be written as follows
\begin{equation*}
\begin{pmatrix}
\Rea_1u_{1,1} & \cdots & \Rea_1u_{n,N}\\
\vdots &  & \vdots\\
\Rea_{N}u_{1,1} & \cdots & \Rea_{N}u_{n,N} 
\end{pmatrix}
\begin{pmatrix}
\alpha_{1,1}\\
\vdots\\
\alpha_{n,N}        
\end{pmatrix}
= \begin{pmatrix}
f(\theta_1)\\
\vdots\\
f(\theta_N)        
\end{pmatrix}.
\end{equation*}
In terms of the reproducing kernel $Q$, the matrix 
\begin{equation*}
\begin{pmatrix}
\Rea_1u_{1,1} & \cdots & \Rea_1u_{n,N}\\
\vdots &  & \vdots\\
\Rea_{N}u_{1,1} & \cdots & \Rea_{N}u_{n,N} 
\end{pmatrix}
\end{equation*}
can also be written in block form
\begin{equation}\label{eq:def-matrix-Q}
Q:=
\begin{pmatrix}
Q(\theta_1,\theta_1) & \cdots & Q(\theta_1,\theta_N)\\
\vdots &  & \vdots\\
Q(\theta_N,\theta_1) & \cdots & Q(\theta_N,\theta_N)
\end{pmatrix}.
\end{equation}
Recall that,
\begin{equation*}
Q(\theta_l,\theta_k) = \int_0^T e^{A_l(T-s)} B_l B_k^\trans\,e^{A_k^\trans(T-s)} \di{s}
\end{equation*}
and thus $Q$ is just the controllability Gramian \eqref{eq:def-Gramian_parallel}, i.e. $Q$ is positiv definite. Hence, \eqref{eq:lin-eq-RNu=FN} is equally written as $Q\alpha=F_N$ and the real numbers $\alpha_{1,1},..., \alpha_{n,N}$ are given by
\begin{equation*}
\begin{pmatrix}
\alpha_{1,1}\\
\vdots\\
\alpha_{n,N}        
\end{pmatrix}
= Q^{-1}
\begin{pmatrix}
f(\theta_1)\\
\vdots \\
f(\theta_N)
\end{pmatrix}.
\end{equation*}
Plugging this into \eqref{eq:def-u-alpha-v}  we get that 
\begin{equation*}
u_{\parallel N} = \sum_{k=1}^{N} \sum_{i=1}^{n} \alpha_{i,k} \, u_{i,k}
\end{equation*}
is another explicit representation of the minimal $L^2_m([0,T])$-norm  solution to \eqref{eq:lin-eq-RNu=FN}. The next result states conditions under which the solution $u_{\parallel N}$ also approximates the minimal $L^2$-norm solution to $\Rea u=f$. Before we state and prove this, we need some notation. For a matrix $M=(M_{ij}) \in \mathbb R^{n \times n}$ we will use the following matrix norm:
\begin{equation*}
\| M\|_{n,\infty} := n \max_{1\leq i,j\leq n} |M_{ij}|
\end{equation*}
and recall that it is consistent and submultiplicative (cf. \cite[Chapter~5]{horn1990matrix}), i.e. it holds
\begin{equation*}
\|Mx\| \leq \| M\|_{n,\infty} \|x\| \quad \text{ and } \quad \| M\,N\|_{n,\infty} = \| M\|_{n,\infty} \,  \| N\|_{n,\infty} 
\end{equation*}
for all $M,N \in \mathbb R^{n \times n}$ and for all $x\in \mathbb R^n$. Moreover, we say that the reproducing kernel $Q$ is Lipschitz continuous on $\p \times \p$ if there is an $L_Q>0$ such that
\begin{equation*}
\| Q(\eta_1,\tau_1) - Q(\eta_2,\tau_2)\|_{n,\infty} \leq L_Q \left( \,|\eta_1 - \eta_2| + |\tau_1- \tau_2|\,\right)
\end{equation*}
for all $(\eta_1,\tau_1), (\eta_2,\tau_2) \in \p\times \p$. Moreover, let
\begin{equation*}
\Delta_{\max} := \max_{k=1,...,N-1} |\theta_{k+1} -\theta_k|.
\end{equation*}
Then, we have the following refinement of the convergence result obtained in \cite[Theorem~3.1]{nashed_math_comp_1974} and \cite{wahba_1973}.

\begin{theorem}\label{thm:nashed_refine}
Let $(A,B) \in C_{n,n}(\p)\times C_{n,m}(\p)$ and $f \in \I \Rea$. 
\begin{enumerate}
 \item[(a)] Then, it holds
 \begin{align*}
  \lim_{\Delta_N\to 0} \| u_{\parallel N} - \Rea^\sharp f\|_{L^2([0,T],\mathbb R^m)} =0.
 \end{align*}
  \item[(b)] If $Q$ is Lipschitz contiunuous on $\p \times \p$ and there is an $g \in L^2_n(\p)$ such that $\Rea^\sharp f= \Rea^*g$ (i.e. $f \in \Rea \Rea^* (L^2_n(\p)$) then, we have
 \begin{align*}
   \| u_{\parallel N} - \Rea^\sharp f\|_{L^2_m([0,T])} =  
   \sqrt{ \Delta_{\max}  \,  3\,n \left( 4\,M_Q e^{-\tfrac{2}{\Delta_{\max}}} + L_Q\, \sqrt{\Delta_{\max}} \right) } \, \|g\|_{L_n^2(\p)},
 \end{align*}
 where  $M_Q = \max_{i,j =1,...,n} \max_{(\tau,\eta) \in \p\times \p}  |Q(\tau,\eta)_{ij}|$.
\end{enumerate}

\end{theorem}

\begin{proof}(a): In a first step we show that $u_{\parallel N}$ satisfies
\begin{equation}\label{eq:1st_prop}
\Rea u_{\parallel N} =  P_{V_N} f,
\end{equation}
where $P_{V_N}$ denotes the orthogonal projection in $\I \Rea$ onto $V_N$.  To see this, we shall show that
\begin{equation*}
\langle \Rea u_{\parallel N} - f, v\rangle_{\I \Rea} = 0 \qquad \forall \, \, v \in V_N.
\end{equation*} 
Recall that, by construction it holds $u_{\parallel N}= I_N^\sharp F_N = P_{U_N} \Rea^\sharp f$, i.e.
\begin{equation*}
\langle u_{\parallel N} - \Rea^\sharp f , u \rangle_{L^2_m([0,T])} = 0 \qquad \forall \, \, u \in U_N.
\end{equation*} 
Let $v \in V_N$. Since $\Rea\colon U_N \to V_N$ is isometric, there is an $u \in U_N$ so that $\Rea u =v$ and since $\R \R^\sharp f=f$ we have
\begin{align*}
\langle \Rea u_{\parallel N} -  f , v \rangle_{\I \Rea} &= \langle \Rea u_{\parallel N} -  f , \Rea u \rangle_{\I \Rea} = \langle \Rea (u_{\parallel N} -  \Rea^\sharp f) , \Rea u \rangle_{\I \Rea} \\
&= \langle u_{\parallel N} - \Rea^\sharp f , u \rangle_{L^2_m([0,T])} =0.
\end{align*} 
Let 
\begin{align*}
V := \operatorname{span}\left\{ v_{1,\tilde \theta},...,v_{n,\tilde \theta}\, \, | \, \, \tilde \theta \in \bigcup_{N \in \mathbb N} \{\theta_1,...,\theta_N\} \right\}. 
\end{align*}
Then, as $Q$ is continuous and $\lim_{N \to \infty} \Delta_N =0$, it follows that $V$ is dense in $\I \Rea$, cf. \cite[Proof of Theorem~17.13]{kress}. Thus, using \eqref{eq:convergence_reproducing_witz} and \eqref{eq:1st_prop} we have
 \begin{align*}
\lim_{\Delta_N \to 0} \| u_{\parallel N} - \Rea^\sharp f\|_{L^2_m([0,T])}  &= 
\lim_{\Delta_N \to 0}  \| \Rea u_{\parallel N} -  f\|_{\I \Rea} \\
&=\lim_{\Delta_N \to 0}  \| \pro_{V_N} f -  f\|_{\I \Rea} =0.
 \end{align*}

(b): The assumption $f \in \Rea \Rea^* (L^2_n(\p))$ implies that there is a function $g=\begin{pmatrix} g_1\\ \vdots \\ g_N \end{pmatrix}                                                                                                                        \in L^2_n(\p)$ such that
  \begin{equation}\label{eq:f_Q_g}
f (\theta) = \int_\p Q(\theta,\eta) g(\eta) \di{\eta}.
\end{equation}
Since for any $\tilde f \in V_N$ it holds
 \begin{align*}
\| f - \pro_{V_N} f \|_{\I \Rea} \leq \|f- \tilde f\|_{\I \Rea},
 \end{align*}
the verification of (b) is simply to find appropriate functions $p_k \colon \p \to \mathbb R$, $k=1,...,N$ such that the function
 \begin{align*}
\tilde f (\cdot) := \sum_{k=1}^N Q(\cdot,\theta_k) c_k, \qquad c_k = \begin{pmatrix}
                                                                     \int_\p p_k(\tau) g_1(\tau) \di{\tau}\\
                                                                     \vdots\\
                                                                     \int_\p p_k(\tau) g_n(\tau) \di{\tau}
                                                                    \end{pmatrix}
 \end{align*}
is in $V_N$ and $\|f- \tilde f\|_{\I \Rea}$ can be bounded from above.

It holds,
 \begin{align*}
\|f- \tilde f\|^2_{\I \Rea} = \langle f, f \rangle_{\I \Rea} -2 \langle  f, \tilde f \rangle_{\I \Rea} + \langle \tilde f , \tilde f \rangle_{\I \Rea}.
 \end{align*}
Next, we consider the terms individually. By assumption it holds $\Rea^\sharp f= \Rea^*g$ and one has
 \begin{align*}
\langle f, f \rangle_{\I \Rea}  &= \langle \Rea^\sharp f, \Rea^\sharp f \rangle_{L^2_m([0,T])} =  \langle \Rea^*g, \Rea^*g \rangle_{L^2_m([0,T])} \\
& = \int_0^T \left( \int_\p B(\tau)^\trans e^{A(\tau)^\trans(T-s)} g(\eta) \di{\tau} \right)^\trans  \left( \int_\p B(\eta)^\trans e^{A(\eta)^\trans(T-s)} g(\eta) \di{\eta} \right)                    \di{s}\\
&=\int_\p   \int_\p g(\tau)^\trans  \left( \int_0^T e^{A(\tau)(T-s)}  B(\tau)   B(\eta)^\trans e^{A(\eta)^\trans(T-s)}  \di{s}\right)           g(\eta)     \di{\tau}    \di{\eta}   \\
&=\int_\p   \int_\p g(\tau)^\trans  Q(\tau,\eta)           g(\eta)     \di{\tau}    \di{\eta}.
 \end{align*}
Furthermore, exploiting the reproducing kernel structure and \eqref{eq:f_Q_g} we get
 \begin{align*}
\langle f, \tilde f \rangle_{\I \Rea}  &= \sum_{k=1}^N \langle f,  Q(\cdot,\theta_k) c_k\rangle_{\I \Rea} 
= \sum_{k=1}^N \langle f(\theta_k),   c_k\rangle_{\mathbb R^n} = \sum_{k=1}^N  c_k^\trans f(\theta_k)   \\
&=\sum_{k=1}^N \int_\p \int_\p p_k(\tau) g(\tau)^\trans Q(\theta_k,\eta) g(\eta) \di{\tau} \di{\eta}
 \end{align*}
and
 \begin{align*}
\langle \tilde f, \tilde f \rangle_{\I \Rea} &= 
\sum_{k=1}^N \sum_{l=1}^N  \left\langle   Q(\cdot,\theta_k) c_k,   Q(\cdot,\theta_l) c_l  \right\rangle_{\I \Rea} 
= \sum_{k=1}^N     \sum_{l=1}^N   \langle    c_k,   Q(\theta_k,\theta_l) c_l \rangle_{\mathbb R^n} \\
&= \int_\p \int_\p \sum_{k=1}^N \sum_{l=1}^N   g(\tau)^\trans p_k(\tau) Q(\theta_k,\theta_l)p_l(\eta) g(\eta) \di{\tau}\di{\eta}.
 \end{align*}
This yields
 \begin{multline*}
\|f- \tilde f\|^2_{\I \Rea} = \\
\int_\p \int_\p  g(\tau)^\trans \hspace{-.1cm}
\left\{ \hspace{-.1cm} Q(\tau,\eta) - \sum_{k=1}^N \Big( 2 \, p_k(\tau)  Q(\theta_k,\eta) - \sum_{l=1}^N 
p_k(\tau) Q(\theta_k,\theta_l)p_l(\eta) \Big)
\right\}
g(\eta) \di{\tau} \di{\eta}.
 \end{multline*}
 
In the next step we construct linear interpolations of the matrices $Q(\eta,\tau)$. To this end, we partition the interval $\p$ according to the moments
 \begin{align*}
\p = [\theta_1,\theta_2] \cup [\theta_2,\theta_3] \cup \cdots \cup [\theta_{N-1},\theta_N] =: \p_1 \cup \cdots \cup \p_{N-1}.
 \end{align*}
Then, on each subinterval $\p_k$ we define the polynomials $l_{k,1}$ and $l_{k,2}$ as follows
\begin{align}
 \begin{split}
  l_{k,1}(\theta) = 
  \frac{\theta - \theta_{k+1}} {\theta_{k}-\theta_{k+1}} 
                \qquad \text{ and } \qquad 
        l_{k,2}(\theta) = 
  \frac{\theta - \theta_{k}} {\theta_{k+1}-\theta_{k}} .         
\end{split}
\end{align}
Thus, on $\p_k$ and for fixed $\eta \in \p$ the linear interpolation of the matrix-valued function $\tau \mapsto Q(\tau,\eta)$ is given by the matrix-valued Lagrange polynomial
\begin{align*}
 \tau  \mapsto B_{1,Q}(\tau,\eta) := l_{k,1}(\tau) \,  Q(\theta_k,\eta) + l_{k,2}(\tau) \, Q(\theta_{k+1},\eta).  
\end{align*}
The notation $B_{1,Q}$ is intended to express that the interpolation polynomial is equal to the first Bernstein polynomial. Hence, for fixed $\eta>0$ and $\tau \in \p_k$, as shown in the proof of Theorem~1 in \cite[Eq.~(4)]{gzyl1997} we get for each entry $Q(\tau,\eta)_{ij}$ the following estimate
\begin{align*}
| Q(\tau,\eta)_{ij} -B_{1,Q}(\tau, \eta)_{ij}| \leq 4 \, e^{- \tfrac{2}{\theta_{k+1} -\theta_k }} \,\max_{\tau \in \p_k} |Q(\tau,\eta)_{ij}| \,  + L_Q \sqrt{\theta_{k+1}-\theta_k},
\end{align*}
where $L_Q>0$ denotes the Lipschitz constant of $Q$. In order to apply the previous estimation, we define for $k=1,...,N$ the functions
 \begin{align}
 \begin{split}
  p_{k}(\theta) = \begin{cases}l_{k,1}(\theta)
  &\quad \theta  \in \p_k\\
                 0 &\quad  \theta \not \in \p_k 
                \end{cases}
             \quad \text{ and }   \quad 
        p_{k+1}(\theta) = \begin{cases}l_{k,2}(\theta)
   &\quad \theta  \in \p_k\\
                 0 &\quad  \theta \not \in \p_k .
                \end{cases}          
\end{split}
\end{align}
Then, for $\tau \in \p_k$ and $\eta \in \p_l$ it holds
 \begin{align*}
 Q(\tau,\eta) &- \sum_{k=1}^N \left( 2 \, p_k(\tau)  Q(\theta_k,\eta) - \sum_{l=1}^N 
p_k(\tau) Q(\theta_k,\theta_l)p_l(\eta) \right)  = \\
Q(\tau,\eta) &-    l_{k,1}(\tau) \, Q(\theta_k,\eta) - l_{k,2}(\tau) \,  Q(\theta_{k+1},\eta) \\[2ex]
&-l_{k,1}(\tau) \Big( Q(\theta_k,\eta) -   l_{l,1}(\eta) Q(\theta_k,\theta_l) - l_{l,2}(\eta) Q(\theta_k,\theta_{l+1}) \Big) \\[2ex]
&-l_{k,2}(\tau) \Big( Q(\theta_{k+1},\eta) -   l_{l,1}(\eta) Q(\theta_{k+1},\theta_l) - l_{l,2}(\eta) Q(\theta_{k+1},\theta_{l+1})\Big). 
\end{align*}
Defining 
\[
 M_Q = \max_{i,j =1,...,n} \max_{(\tau,\eta) \in \p\times \p}  |Q(\tau,\eta)_{ij}|
\]
one has for each entry the estimate
 \begin{align*}
 &\left| Q(\tau,\eta)_{ij}  - \sum_{k=1}^N  \left( 2 \,  p_k(\tau)  Q(\theta_k,\eta)_{ij}   - \sum_{l=1}^N 
 p_k(\tau) Q(\theta_k,\theta_l)_{ij} p_l(\eta)\right)\right|  \\ 
&\leq | Q(\tau,\eta)_{ij} -    l_{k,1}(\tau) \, Q(\theta_k,\eta)_{ij} - l_{k,2}(\tau) \,  Q(\theta_{k+1},\eta)_{ij}|\\
&\quad + | l_{k,1}(\tau) |\, | Q(\theta_k,\eta)_{ij} -   l_{l,1}(\eta) Q(\theta_k,\theta_l)_{ij} - l_{l,2}(\eta) Q(\theta_k,\theta_{l+1} )_{ij} |\\
  &\quad+ | l_{k,2}(\tau) |\, |  Q(\theta_{k+1},\eta)_{ij} -   l_{l,1}(\eta) Q(\theta_{k+1},\theta_l)_{ij} - l_{l,2}(\eta) Q(\theta_{k+1},\theta_{l+1} )_{ij}|\\
&\leq 12\,e^{-\tfrac{2}{\Delta_{\max}}   } \,M_Q\,  + 3\,L_Q \,\sqrt{\Delta_{\max}}.
\end{align*}
That is, using $C= 3\,n \left( 4\,M_Q e^{-\tfrac{2}{\Delta_{\max}}   } + L_Q \,\sqrt{\Delta_{\max}} \right)$ for a moment, we have
 \begin{align*}
 \Bigg\|  Q(\tau,\eta) &- \sum_{k=1}^N \left( 2 \, p_k(\tau)  Q(\theta_k,\eta) - \sum_{l=1}^N 
p_k(\tau) Q(\theta_k,\theta_l)p_l(\eta) \right) \Bigg\|_{n,\infty} 
\leq C.
\end{align*}
Then, using the Cauchy-Schwarz inequality we get 
 \begin{align*}
\|f- \tilde f\|^2_{\I \Rea} &\leq 
\sum_{\nu=1}^N  \int_{\p_\nu} \int_{\p_\nu}  \|g(\tau)\|_{\mathbb R^n} \,C\, \|g(\eta)\|_{\mathbb R^n}  \di{\tau} \di{\eta}\\
& \leq \sum_{\nu=1}^N\,C  
\left( \int_{\p_\nu}\, \|g(\tau)\|_{\mathbb R^n}^2\,\di{\tau} \right)^{\frac{1}{2}}
 \left( \int_{\p_\nu}\, 1\,\di{\eta} \right)   \left( \int_{\p_\nu}\, \|g(\eta)\|_{\mathbb R^n}^2\,\di{\eta} \right)^{\frac{1}{2}}  \\
& \leq  C\, \Delta_{\max} \,  \sum_{\nu=1}^N   \left(\int_{\p_\nu}   \|g(\tau)\|_{\mathbb R^n}^2 \di{\tau}\right)   \leq C\, \Delta_{\max} \|g\|_{L_n^2(\p)}^2
 \end{align*}
and finally
 \begin{align*}
\|f- \tilde f\|^2_{\I \Rea}  \leq    \Delta_{\max}  \,  3\,n \left( 4\,M_Q e^{-\tfrac{2}{\Delta_{\max}}} + L_Q\, \sqrt{\Delta_{\max}} \right)  \, \|g\|^2_{L_n^2(\p)}.
 \end{align*}
This shows the assertion.
\end{proof}

\medskip

Note that, since we neither used the particular kernel $K(\theta,s)= e^{A(\theta)(T-s)}B(\theta)$ of the reachability operator nor that the operator is compact, Theorem~\ref{thm:nashed_refine} is valid for general integral equations of first kind with a continuous kernel $K(\theta,s)$ and a Lipschitz continuous reproducing kernel 
\begin{equation*}
 Q(\theta,\tilde \theta) = \int_0^T K(\theta,s) K(\tilde \theta,s)^\trans \di{s}.
\end{equation*}
We will now formulate a different version of the result, which is more convenient for the purpose of this paper.  Since  for $f \in \I \Rea$ it holds $ \Rea \Rea^\sharp f=f$ one has
\begin{equation*}
 \|\Rea u_N -f\|_{\infty} \leq \|\Rea\|\,  \|u_N - \Rea^\sharp f\|_{L^2_m([0,T])},
\end{equation*}
where $\|\mathcal{R}\|$ denotes the operator norm of the reachability operator $\Rea$ defined in \eqref{eq:def_reachability_operator}. Then, for given $\e>0$ and appropriate $f$ Theorem~\ref{thm:nashed_refine} yields the subsequent lower bound on the number of moment parameters that are required such that $u_{\parallel N}$ is a suitable input. 

\begin{theorem}\label{thm:ensemble_nashed}
Suppose that $(A,B) \in C^1_{n,n}(\p) \times C^1_{n,m}(\p)$ is uniformly ensemble controllable. Let $\e>0$ and $f\in \I \Rea$ such that there is a function $g \in L_n^2(\p)$ with $\Rea^\sharp f = \Rea^*g$. Then, the input $u_{\parallel N}$ defined in \eqref{eq:u_parallel} satisfies 
\begin{equation*}
 \|\Rea u_{\parallel N} - f\|_{\infty} \leq \e
\end{equation*} 
if the moments $\theta_1 < \cdots < \theta_N$ are chosen such that 
$$\Delta_{\max} < \tfrac{\e^2}{\|\Rea\|^2\, \|g\|_{L_n^2(\p)}^2 \, 3\,n \,\big( 4\,M_Q e^{-\tfrac{2}{\theta_N-\theta_1}} + L_Q \sqrt{|\theta_N- \theta_1|} \big) } .$$
\end{theorem}

\section{Consensus methods}\label{sec:consensus}

In this section, we present additional methods to obtain solutions to the linear equation
\begin{equation*}
\mathcal{I}_N u = F_N .
\end{equation*}
Recall that uniform ensemble controllability implies that each finite-dimensional moment system $(A_i,B_i), i=1,...,N$ is controllable, cf. \cite{helmke2014uniform}. Thus, for each $f_i:=f(\theta_i)$ there is an open-loop input $u_i \in L^2_m([0,T])$ steering the origin to $f_i$. Using the controllability Gramian
 \begin{equation*}
W_i = \int_0^T e^{A_i\,(T-s)} B_i B_i^\trans e^{A_i^\trans\,(T-s)} \, \di{s}
\end{equation*}
the $L^2$-norm minimal input function $u_i$ is given by
 \begin{equation}\label{eq:moment-input}
u_i = B_i^\trans e^{A_i^\trans\,(T-\cdot)}\, W_i^{-1} f_i.
\end{equation}
We will now use the moment input functions $u_i$, $i=1,...,N$ to derive a global input function which is an approximate solution to \eqref{eq:lin_op_eq}. We note that the subsequent methods do not require the inversion of the matrix $Q$ defined in \eqref{eq:def-matrix-Q}.

\subsection{Projected gradient method}

In order to define our method, we consider for $k=1,...,N$ the operators $\R_k \colon L^2_m([0,T]) \to \mathbb{R}^{n}$,
\begin{equation*}
\R_k u  =
\int_0^T  e^{A_k(T-s)} B_k\, u(s) \di{s}
\end{equation*}
and their adjoints  $\R_k^* \colon \mathbb{R}^{n}  \to L^2_m([0,T])$,
\begin{equation*}
\R_k^*v= B_k^\trans e^{A^\trans_k(T-s)}v.
\end{equation*}
For $i=1,...,N$ we consider the system of differential equations
\begin{align}\label{eq:Bolte_i}
\dot u_i(t) = \frac{1}{N} \sum_{k=1}^N \big( u_k(t)-u_i(t))  - \frac{1}{N} \sum_{k=1}^N \R_k^*\R_k u_k(t) - \R_k^*f_k
\end{align}
defined on the Hilbert space $L^2_m([0,T])$. As initial values $u_i(0)$ we pick the known moment control input functions $u_{i}$ defined in \eqref{eq:moment-input}. The first part of the right-hand side is a consensus term and the second part is a gradient term so that $u_i(t)$ converges to a solution of $\R_i u=f_i$.  Before stating and proving the convergence result, we first recall that a sequence $(u_n)_{n\in \mathbb N}$ in $ L^2_m([0,T])$ is said to converge weakly to $u$ if for all $v \in L^2_m([0,T])$ it holds
\begin{equation*}
\lim_{n\to \infty} \langle u_n ,  v \rangle_{L^2_m([0,T])} = \langle u,  v \rangle_{L^2_m([0,T])}.
\end{equation*}
We shortly write $u_n\rightharpoonup u$ if $(u_n)_{n\in \mathbb{N}}$ converges weakly to $u$. The convergence result is as follows.

\medskip
\begin{theorem}
For each $i=1,...,N$ the solution of the differential equation \eqref{eq:Bolte_i} exists for all $t\geq 0$ and converges weakly to a global solution, i.e. for all $i=1,...,N$ one has
\begin{equation*}
 u_i(t) \rightharpoonup u \qquad \text{ as }\quad  t \to \infty
\end{equation*}
for some $u \in L^2_m([0,T])$. In particular, we have 
$$\mathcal{I}_Nu=
\begin{pmatrix}
\R_1 u\\
 \vdots\\
\R_N u\end{pmatrix}
=
F_N.$$
\end{theorem}
\medskip
{\it Proof.} First we define the consensus subspace
\begin{equation*}
C=\{ u=(u_1,...,u_N) \, |\, u_1=\cdots = u_N \},
\end{equation*}
which is nonempty, closed and convex. Furthermore, let $V \colon L^2_m([0,T])^N \to [0,\infty)$,
\begin{equation*}
V(u) = \sum_{k=1}^N \frac{1}{2} \| \R_k u_k - f_k\|_{\mathbb{R}^n}^2. 
\end{equation*}
The functional $V$ is Frech\'{e}t differentiable with
\begin{equation*}
\nabla V(u) = 
\begin{pmatrix}
\R_1^*\R_1 u_1 - \R_1^*f_1\\
\vdots\\
\R_N^*\R_N u_N - \R_N^*f_N
\end{pmatrix}.
\end{equation*}
Next, we note that the orthogonal projection $P_C \colon L^2_m([0,T])^N \to C$ is given by
\begin{equation*}
P_C\,u= 
\begin{pmatrix}
\tfrac{1}{N} ( u_1 +  \cdots + u_N) \\
\vdots\\
\tfrac{1}{N} ( u_1 +  \cdots + u_N) 
\end{pmatrix}.
\end{equation*}
Using this notation and stacking the equations \eqref{eq:Bolte_i} for $i=1,...,N$ we get
\begin{equation}\label{eq:Bolte_global}
\dot u(t) = -u(t) + P_C\left( u(t) - \nabla V\big( u(t)\big) \right).
\end{equation}
The assertion then follows from \cite[Theorem~3.2]{bolte_JOTA_2003}. 
\hfill $\Box$

\medskip

In a next step we modify the system of differential equations \eqref{eq:Bolte_i} to get strong convergence of the corresponding trajectories. To this end, let $\eta:[0,\infty) \to [0,\infty)$ be a nonincreasing continuously differentiable function with bounded derivative satisfying 
\begin{equation}\label{eq:eta}
\lim_{t\to \infty} \eta(t) =0, \qquad \int_0^\infty \eta(t) \di{t}= \infty.
\end{equation}
Then, for $i=1,...,N$ we consider the modified ordinary differential equations
\begin{align}\label{eq:Bolte_i-eta}
\dot u_i(t) = \frac{1}{N} \sum_{k=1}^N \Big( \big(1-\eta(t)\big)u_k(t)-u_i(t) \Big)  - \frac{1}{N} \sum_{k=1}^N \R_k^*\R_k u_k(t) - \R_k^*f_k.
\end{align}
As initial conditions we pick one of the moment input functions given in \eqref{eq:moment-input}, i.e 
$$u_i(0)= u_0  \in \left\{u_1(s) = B_1^\trans e^{A_1^\trans\,(T-s)}\, W_1^{-1} f_1,..., u_N(s) = B_N^\trans e^{A_N^\trans\,(T-s)}\, W_N^{-1} f_N\right\}.$$
Collocating for $i=1,...,N$, the equations \eqref{eq:Bolte_i-eta} can be written as
\begin{equation}\label{eq:Bolte_global-eta}
\dot u(t) = -u(t) + P_C\left( u(t) - \nabla V\big( u(t)\big) -\eta(t)\,u(t)\right).
\end{equation}
Using \cite[Theorem~5.1]{bolte_JOTA_2003} we conclude the following strong convergence result.
\medskip
\begin{theorem}
For each $i=1,...,N$ the solution of the differential equation \eqref{eq:Bolte_i-eta} exists for all $t\geq 0$ and converges strongly to a global solution, i.e. for all $i=1,...,N$ one has
\begin{equation*}
 \lim_{t\to  \infty} u_i(t)=  u,
 \end{equation*}
for some $u\in L^2_m([0,T])$. In particular, we have 
$\mathcal{I}_Nu=
F_N$. 
\end{theorem}
\medskip

\subsection{Averaging Method}

Another method to solve the infinite-dimensional linear integral equation
\begin{equation*}
\mathcal{I}_N u = F_N
\end{equation*}
using the known moment input functions $u_i \in M_i$ is inspired by \cite{shi2017network}, where the authors present distributed algorithms for solving finite-dimensional linear equations. Let
\begin{equation*}
M_i:=\{ u \in  L^2_m([0,T])\, \, |\, \R_i u = f_i \, \}
\end{equation*}
denote the set of all solutions to 
\begin{equation}\label{eq:R_k}
 \R_i u = f_i.
\end{equation}
Following the ideas for distributed solving of linear finite-dimensional equations, cf. \cite{mou2013fixed,shi2017network}, we consider the orthogonal projection $P_{M_i} \colon L^2_m([0,T]) \to M_i$. In order to compute $P_{M_i}$, we note that $M_i$ is an affine subspace and so $P_{M_i}$ is given by
\begin{align*}
g \mapsto P_{M_i} g &= u_i + P_{\op{ker} \R_i} (g-u_i)\\
&=u_i + (I - P_{(\op{ker} \R_i)^\perp}) (g-u_i),
\end{align*}
with $u_i \in M_i$ defined in \eqref{eq:moment-input}. Since the range of $\R_i$ is finite-dimensional, it is also closed. Then, by standard functional analysis one has 
\[(\op{ker} \R_i)^\perp = \I\R_i^*.\] 
Consequently, it remains to determine the orthogonal projection onto $\I\R_i^*$. This is derived as follows. Let $h \in L_m^2([0,T])$ and $y = P_{\I\R_i^*}h$. Then, there is a $v_y\in \mathbb{R}^n$ such that 
\begin{equation*}
 y(s) = B_i^\trans e^{A_i^\trans(T-s)}v_y, \quad \forall \, s\in [0,T].
\end{equation*}
Furthermore, it holds
\begin{equation*}
\langle h- y ,g\rangle_{L^2_m([0,T])} = 0 \qquad \forall \, g \in \I\R_i^*,
\end{equation*}
or equivalently
\begin{equation*}
\langle h-B_i^\trans e^{A_i^\trans(T-\cdot)}v_y   , B_i^\trans e^{A_i^\trans(T-\cdot)}v \rangle_{L^2_m([0,T])} = 0 \qquad \forall \, v \in \mathbb{R}^n.
\end{equation*}
It follows that
\begin{equation*}
v_y= W_i^{-1} \int_0^T e^{A_i(T-s)}B_i h(s)  \di{s}
\end{equation*}
and therefore, we obtain
\begin{align*}
P_{(\op{ker} \R_i)^\perp} h (s) &=  B_i^\trans e^{A_i^\trans(T-s)}W_i^{-1} \int_0^T e^{A_i(T-\tau)}B_i h(\tau)  \di{      \tau}\\
 &= \R_i^*(W_i^{-1} \R_i h).
\end{align*}
As a result, the orthogonal projection onto $M_i$ is given by
\begin{equation}\label{eq:projection-M_i}
P_{M_i} g =g -  \R_i^* (W_i^{-1}\R_i g) + \R^*_i ( W_i^{-1}f_i).
\end{equation}  
A simple calculation shows that the orthogonal projection onto $M_i$ is affine linear, i.e.
\begin{equation*}
P_{M_i} (g+h)  = P_{M_i} g + P_{M_i} h - P_{M_i} (0)
\end{equation*}  
and the following identity holds
\begin{equation*}
P_{M_i} \left(\frac{g_1+\cdots + g_N}{N}\right)   =  \frac{ P_{M_i} g_1 +\cdots + P_{M_i} g_N}{N}.
\end{equation*}
Now, we are in the position to present our method. In order to obtain a global solution we define the projection-consensus system of differential equations
\begin{equation}\label{eq:def-proj+consensus}
\dot u_i(t) = \sum_{k=1}^N \left( P_{M_i} u_k(t)  -  u_i(t) \right).
\end{equation}
As initial conditions we chose again the known moment solutions $u_i(0)=u_i \in M_i$ defined in \eqref{eq:moment-input}. The following result shows that the mean-value of the moment solutions converges and its limit is a global solution. More precisely, we have the following

\medskip
\begin{theorem}
The solutions to \eqref{eq:def-proj+consensus} satisfy
\begin{equation*}
\lim_{t\to \infty} \frac{1}{N} \sum_{i=1}^N u_i(t) = u,
\end{equation*}
and $u$ is a solution to $\mathcal{I}_N u =F_N$.
\end{theorem}
\medskip

{\it Proof.} We begin with denoting the mean-value by $z(t)$, i.e.
\begin{equation*}
z(t):= \frac{1}{N} \sum_{i=1}^N u_i(t).
\end{equation*}
The evolution of the mean-value is given by the following ordinary differential equation
\begin{align*}\label{eq:dot-z}
\dot z(t) &= \frac{1}{N}  \sum_{i=1}^N \dot u_i(t)  = \frac{1}{N}  \sum_{i=1}^N \left(    \sum_{k=1}^N P_{M_i} u_k(t)  -  u_i(t) \right)\\
&=  \sum_{i=1}^N \left(   \frac{1}{N}  \sum_{k=1}^N P_{M_i} u_k(t)      -   \frac{1}{N} \sum_{l=1}^N   u_l(t) \right)\\
&=  \sum_{i=1}^N \left(    P_{M_i}\left(\frac{1}{N}\sum_{k=1}^N   u_k(t)\right)      -   \frac{1}{N} \sum_{l=1}^N   u_l(t) \right)\\
&=  \sum_{i=1}^N \left(    P_{M_i} z(t)     -   z(t) \right).
\end{align*}
Spelling out the orthogonal projection \eqref{eq:projection-M_i} the latter differential equation also takes the form
\begin{equation}\label{eq:NR}
\dot z(t) =  - \sum_{i=1}^N  \R_i^* \left(W_i^{-1} \R_i z(t)\right) + \sum_{i=1}^N  \R_i^* \left(W_i^{-1} f_i \right).
\end{equation}
Since $W_i^{-1}$ is symmetric and positiv definite, by \cite[Theorem~7.2.6]{horn1990matrix} there is a symmetric and positiv definite matrix $S_i$ such that $S_i\cdot S_i =W_i^{-1}$. Further, we consider the modified interpolation operator $\widetilde{\mathcal{I}}_N \colon L^2_m([0,T]) \to \mathbb{R}^{nN}$,
 \begin{equation*}
\widetilde{\mathcal{I}}_N u  =
\begin{pmatrix}
S_1 \int_0^T  e^{A_1(T-s)} B_1\, u(s) \di{s}\\
\vdots\\
S_N \int_0^T  e^{A_N(T-s)} B_N\, u(s) \di{s}
\end{pmatrix}
\end{equation*}
and its adjoint  $\widetilde{\mathcal{I}}_N^* \colon \mathbb{R}^{nN}  \to L^2_m([0,T])$,
\begin{equation*}
v= \begin{pmatrix}
v_1\\
\vdots\\
v_N
\end{pmatrix}
\mapsto
\widetilde{\mathcal{I}}_N^*v (s)=  \sum_{i=1}^N B_i^\trans e^{A_i^\trans (T-s)} S_i v_i.
\end{equation*}
Using $\widetilde F_n:= \begin{pmatrix} S_1 f_1 & \cdots & S_N f_N \end{pmatrix}^\trans$, the linear differential equation \eqref{eq:NR} also reads as
\begin{equation*}
\dot z(t) =  -   \widetilde{\mathcal{I}}_N^* \widetilde{\mathcal{I}}_N  z(t)  + \widetilde{\mathcal{I}}_N^* \widetilde F_N.
\end{equation*}
We note that for $V \colon L^2_m([0,T]) \to [0,\infty)$,
\begin{equation*}
V( z) := \frac{1}{2} \| \widetilde{\mathcal{I}}_N  z  - \widetilde F_N \|_{\mathbb{R}^{nN}}^2
\end{equation*}
one has
\begin{equation*}
\dot z(t) = - \nabla V(z) = -   \widetilde{\mathcal{I}}_N^* \widetilde{\mathcal{I}}_N  z  + \widetilde{\mathcal{I}}_N^* \widetilde F_N .
\end{equation*}
Thus, applying \cite[Theorem~3.1]{neuberger_LNM} we conclude that $\lim_{t\to \infty} z(t)=u$ exists and satisfies $ \widetilde{\mathcal{I}}_N u = \widetilde F_N$, or equivalently
\begin{equation*}
 \int_0^T  e^{A_i(T-s)} B_i\, u(s) \di{s}= f_i \qquad \text{ for all } i=1,...,N.
\end{equation*}
This shows the assertion.
\hfill $\Box$

\providecommand{\href}[2]{#2}
\providecommand{\arxiv}[1]{\href{http://arxiv.org/abs/#1}{arXiv:#1}}
\providecommand{\url}[1]{\texttt{#1}}
\providecommand{\urlprefix}{URL }

\end{document}